\documentclass[a4paper,11pt]{amsart}

\usepackage{amsmath,amssymb,amsthm,latexsym,amscd,mathrsfs}
\usepackage{indentfirst}
\usepackage{stmaryrd}
\usepackage{graphicx}
\usepackage{extarrows}
\usepackage{multirow}
\usepackage{latexsym}
\usepackage{amsfonts}
\usepackage{color}
\usepackage{pictexwd,dcpic}
\usepackage{graphicx}
\usepackage{psfrag}
\usepackage{comment}
\usepackage{enumerate}
\usepackage{hyperref}

\newtheorem{theorem}{Theorem}[section]

\newtheorem{corollary}[theorem]{Corollary}

\newtheorem{lemma}[theorem]{Lemma}

\newtheorem{proposition}[theorem]{Proposition}

\newtheorem{remark}[theorem]{Remark}

\theoremstyle{definition}

\newcommand{\Q}{\mathbb{Q}}
\newcommand{\R}{\mathbb{R}}
\newcommand{\Sp}{\mathbb{S}}
\newcommand{\Hi}{\mathbb{H}}
\newcommand{\RP}{\mathbb{R}\mathbb{P}}

\DeclareMathOperator{\otr}{\mathrm{tr}}
\DeclareMathOperator{\oImag}{\mathrm{Image}}
\DeclareMathOperator{\oRic}{\mathrm{Ric}}

\DeclareMathOperator{\oSpec}{\mathrm{Spec}}

\title{Hypersurfaces with Constant Ricci Eigenvalues in Real Space Forms}

\begin{document}

\author[J. Q. Ge]{Jianquan Ge}
\address{School of Mathematical Sciences, Beijing Normal University, Beijing 100875, P. R. China}
\email{jqge@bnu.edu.cn}
\author[Y. Y. Zhao]{Yuyang Zhao$^{*}$}
\address{School of Mathematical Sciences, Beijing Normal University, Beijing 100875, P. R. China}
\email{yyzhao24@mail.bnu.edu.cn}

\thanks {$^{*}$ the corresponding author.}
\date{}
\thanks{J. Q. Ge is partially supported by the NSFC (No. 12571049) and the Fundamental Research Funds for the Central Universities.}

\begin{abstract}
    The classification of curvature homogeneous hypersurfaces in real space forms was established by Tsukada in 1988, with the remaining rank-two cases in $\Sp^4$ and $\Hi^4$ settled by Bryant-Florit-Ziller in 2025. It is obvious that curvature homogeneity implies constant Ricci eigenvalues. In this paper, we prove that for hypersurfaces in real space forms, the converse also holds: a connected hypersurface immersed in real space forms has constant Ricci eigenvalues if and only if it is curvature homogeneous. Hence, hypersurfaces with constant Ricci eigenvalues in real space forms are also classified, which, in particular, generalizes the classification of Einstein hypersurfaces obtained by Lawson for minimal hypersurfaces and by Ryan for general cases in 1969. Moreover, as a byproduct, curvature inhomogeneous Riemannian manifolds with constant Ricci eigenvalues can not be isometrically immersed in any real space form of codimension one. Finally, we show that a hypersurface with constant Ricci eigenvalues is isoparametric if it is a complete hypersurface either in $\Sp^{n+1}$ or nonflat in $\R^{n+1}$ for $n \geq 3$; or if it is not of constant sectional curvature $-1$ in $\Hi^{n+1}$ for $n \geq 5$.
\end{abstract}

\maketitle

\section{Introduction}
A Riemannian manifold $M^n$ is called \emph{curvature homogeneous} if,
for any two points $p, q \in M$, there exists a linear isometry
$J_{pq} \colon T_pM \to T_qM$ satisfying $J_{pq}^{\ast}R_q = R_p$,
where $R$ denotes the Riemann curvature tensor. Every locally homogeneous space is curvature homogeneous, but the converse fails in general. Singer \cite{Singer1960} proved that on a complete and
simply connected Riemannian manifold, if sufficiently many covariant derivatives of $R$ match, then the metric is in fact
homogeneous. In dimension $n=3$, the curvature tensor is determined by the Ricci tensor, so curvature homogeneity is equivalent to the constancy of the Ricci eigenvalues. In higher dimensions $n\geq 4$, curvature homogeneity still  obviously implies that all Ricci eigenvalues are constant, but the converse does not hold in general (see counterexamples in \cite{GZ-Counter}). However, for hypersurfaces in real space forms, we will show that the converse indeed holds. 

For hypersurfaces $M^n$ in real space forms $\Q_c^{n+1}$ of constant sectional curvature $c$, Tsukada \cite{Tsukada1988} established a complete characterization: \emph{a hypersurface $M^n \subset \Q_c^{n+1}$ $(n \geq 3)$ is curvature homogeneous if and only if it is isoparametric, or has constant sectional curvature $c$, or has rank-two shape operator with constant scalar curvature.}
Here, a hypersurface is called \emph{isoparametric} if all its principal curvatures are constant. The study of isoparametric
hypersurfaces in real space forms was initiated by Levi-Civita, Segre and \'{E}lie Cartan in the 1930s and has since become a central subject in submanifold geometry. The classification was already completed by Levi-Civita and Segre in the Euclidean case and by Cartan in the hyperbolic case, whereas in the unit sphere $\Sp^{n+1}$ the problem is far more intricate and was settled only recently, the final case by Quo-Shin Chi \cite{Chi2020} in 2020. For the classifications and developments of isoparametric hypersurfaces, see the excellent monograph of Cecil and Ryan \cite{CR2015} and recent surveys \cite{Chi2018, GQTY2025}.

Among the three cases in Tsukada's characterization, the rank-two case is the most subtle. In Euclidean space, the known examples are cylinders over constant curvature surfaces. For dimensions greater than $3$, Tsukada \cite{Tsukada1988} 
proved the existence of a unique complete four-dimensional rank-two example in $\Hi^5$, while the three-dimensional cases in $\Sp^4$ and $\Hi^4$ were not covered. 
A decisive advance was made recently in 2025 by Bryant, Florit, and Ziller \cite{BFZ2025}, who applied the Gauss parametrization to obtain a complete classification of the remaining rank-two hypersurfaces in these two four-dimensional space forms. We refer to the hypersurfaces arising in their classification as \emph{Bryant-Florit-Ziller rank-two hypersurfaces}; see Theorem~\ref{thm:BFZhyper}.

The present paper is motivated by the following natural question: what are the hypersurfaces $M^n \subset \mathbb{Q}_c^{n+1}$ whose Ricci tensor has constant eigenvalues? As noted above, this property is implied by curvature homogeneity, and for $n = 3$ the two conditions are equivalent. The simplest instance, where all Ricci eigenvalues coincide, is precisely the \emph{Einstein condition}. 
The classification of Einstein hypersurfaces in real space forms was completed independently by Lawson \cite{Lawson1969} and Ryan \cite{Ryan1969} in 1969.
In the spherical case, Lawson \cite{Lawson1969} proved a rigidity theorem for minimal Einstein hypersurfaces: \emph{if $M^n\subset \Sp^{n+1}$ is minimal Einstein, then either it is totally geodesic, or $n=2k$ and it is an open submanifold of the Clifford torus}
$\Sp^{k}(\frac{1}{\sqrt{2}})\times \Sp^{k}(\frac{1}{\sqrt{2}})
\subset \Sp^{n+1}.$
Without the minimality assumption, Ryan \cite{Ryan1969} proved that \emph{an Einstein hypersurface $M^n \subset \Q_c^{n+1}$ with Ricci curvature $\rho$ falls into exactly one of the following three types:
\begin{itemize}
  \item if $\rho > (n-1)c$, then $M$ is totally umbilical of constant sectional curvature $\frac{\rho}{(n-1)}$;
  \item if $\rho = (n-1)c$, then $M$ has constant sectional curvature $c$;
  \item if $\rho < (n-1)c$, then necessarily $c > 0$ and $\rho = (n-2)c$, and $M$ is locally isometric to $\Sp^p\!\left(\sqrt{\tfrac{p-1}{(n-2)c}}\right) \times \Sp^{n-p}\!\left(\sqrt{\tfrac{n-p-1}{(n-2)c}}\right)$ for some integer $1 < p < n-1$.
\end{itemize}}
A key observation is that every Einstein hypersurface in $\Q_c^{n+1}$
is either of constant sectional curvature $c$ or isoparametric with at most two distinct principal curvatures; in particular, it is curvature homogeneous. Since the Einstein condition corresponds to the case of a single distinct Ricci eigenvalue and is therefore fully understood, it remains to investigate the case of $k \geq 2$ distinct constant Ricci eigenvalues. 

In this paper, we prove the following main theorem:
\begin{theorem}\label{thm:main}
    Let $M^n$ be an $n(\geq 3)$-dimensional connected hypersurface immersed in a real space form $\Q_c^{n+1}$  of constant sectional curvature $c$. Then $M^n$ has constant Ricci eigenvalues if and only if it is curvature homogeneous.
\end{theorem}

An immediate corollary is the following nonexistence rigidity of isometric immersions (see examples in \cite{GZ-Counter}).
\begin{corollary}
   Let $M^n$ be an $n(\geq 4)$-dimensional Riemannian manifold with constant Ricci eigenvalues but not curvature homogeneous. Then there is no isometric immersion of $M^n$ in any real space form $\Q_c^{n+1}$.
\end{corollary}

Applying the results in \cite{Tsukada1988} and \cite{BFZ2025},  by Theorem \ref{thm:main}, hypersurfaces in real space forms with constant Ricci eigenvalues can be classified specifically in the following Theorems \ref{thmR}-\ref{thm:BFZhyper}.
\begin{theorem}\cite{Tsukada1988}\label{thmR}
    Let $M^n$ be an $n(\geq 3)$-dimensional connected curvature homogeneous space and let $f$ be an isometric immersion of $M^n$ into $\R^{n+1}$. Then one of the following holds:
    \begin{enumerate}
        \item $M^n$ is a flat manifold.
        \item $M^n$ is locally isometric to $\Sp^r
        _{\kappa}\times \R^{n-r}$, $3\leq r\leq n$, for a round sphere of constant curvature $\kappa>0$, and $f$ is locally congruent to the standard isometric embedding $\tilde f$ of $\Sp^r_{\kappa}\times \R^{n-r}$ into $\R^{n+1}$.
        \item $M^n$ is locally isometric to $M^2_{\kappa}\times \R^{n-2}$, $\kappa\neq 0$, and $f$ is locally congruent to the product immersion $f_1\times f_2$, where $M^2_{\kappa}$ denotes a surface of constant curvature $\kappa\ (\neq 0)$ and $f_1$ is an isometric immersion of $M^2_{\kappa}$ into $\R^3$ while $f_2$ is the identity map of $\R^{n-2}$ onto $\R^{n-2}$.
    \end{enumerate}
\end{theorem}

\begin{theorem}\cite{Tsukada1988}\label{thmS}
    Let $M^n$ be an $n(\geq 4)$-dimensional connected curvature homogeneous space and let $f$ be an isometric immersion of $M^n$ into $\Sp^{n+1}$. Then either $M^n$ has constant sectional curvature $1$, or $f$ has constant principal curvatures.
\end{theorem}

\begin{theorem}\cite{Tsukada1988}\label{thmH}
    Let $M^n$ be an $n(\geq 4)$-dimensional connected curvature homogeneous space and let $f$ be an isometric immersion of $M^n$ into $\Hi^{n+1}$. Then one of the following holds:
    \begin{enumerate}
        \item $M^n$ is a Riemannian manifold of constant curvature $-1$.
        \item $M^n$ is a Riemannian manifold of constant curvature $\kappa>-1$ and $f$ is totally umbilical.
        \item $M^n$ is locally isometric to $\Sp^r_{c_1}\times \Hi^{n-r}_{c_2}$, $1\leq r\leq n-1$,
        $\frac{1}{c_1}+\frac{1}{c_2}=-1$, $c_1>0$, $c_2<0$ and $f$ is locally congruent to the standard isometric embedding $\tilde f$ of $\Sp^r_{c_1}\times \Hi^{n-r}_{c_2}$ into $\Hi^{n+1}$.
        \item $n=4$ and $M^4$ is locally isometric to the example constructed in \cite[Section~4]{Tsukada1988} and $f$ is locally congruent to the corresponding isometric embedding therein.
    \end{enumerate}
\end{theorem}


The remaining case in curvature homogeneity is the Bryant-Florit-Ziller rank-two hypersurfaces in $\Sp^4$ and $\Hi^4$, classified as follows.
\begin{theorem}\label{thm:BFZhyper}\cite{BFZ2025}
    Let ${\mathcal M}$ be the set of immersed rank-two hypersurfaces in $\Q_c^4~(c=\pm 1)$, whose induced metric has constant scalar curvature. Then ${\mathcal M}$ contains $f_c$ as the only complete example, an isolated hypersurface $\hat f_c$ with a circle of symmetries, and a one parameter family of hypersurfaces admitting no continuous symmetries.
    Moreover, up to a covering, any connected hypersurface in ${\mathcal M}$ is an open subset of one of these, provided it has no leaf of relative nullity of minimal points in the case $c=1$.
\end{theorem}

We now describe the two complete examples in Theorem~\ref{thm:BFZhyper} explicitly. The spherical example $f_1$ is the unit normal bundle of the Veronese surface $\RP^2_{1/3}\subset\Sp^4$, which is itself a homogeneous isoparametric hypersurface with three distinct principal curvatures.
The hyperbolic example $f_{-1}$ is the unit normal bundle of the flat
torus $g=(g_0,1)\colon T^2\to\Sp^4_{-1}\subset\R^{4,1}$, where
$g_0\colon T^2\to\Sp^3(\sqrt{2})\subset\R^4\times\{0\}\subset\R^{4,1}$ is the minimal equivariant Clifford torus; this hypersurface is non-isoparametric and admits a two parameter family of symmetries inherited from the Clifford torus. The remaining example $\hat f_c$ and the one parameter family without continuous symmetries are described in detail in \cite[Sections~5 and~6]{BFZ2025}.

Next we address the question of when a hypersurface with constant Ricci eigenvalues is in fact isoparametric. Using the Gauss equation and the continuity of the principal curvatures, we obtain the following characterization.
\begin{theorem}\label{thm:cmc}
    Let $M^n$ be an $n(\geq 3)$-dimensional connected hypersurface immersed in a real space form $\Q_c^{n+1}$ with constant Ricci eigenvalues. Then $M^n$ is isoparametric if and only if the mean curvature $H$ of $M^n$ is constant.
\end{theorem}

Theorem \ref{thm:cmc} gives strong support to the famous Chern Conjecture (even for the local version): A closed minimal hypersurface with constant scalar curvature in $\mathbb{S}^{n+1}$ is isoparametric (see an important progress by Tang-Wei-Yan \cite{T-W-Y} and Tang-Yan \cite{T-Y}).

Our main Theorem~\ref{thm:main} also yields some sufficient conditions:
\begin{corollary}\label{cor:Ric-isopara}
      Let $M^n$ be an $n(\geq 3)$-dimensional connected hypersurface immersed in  a real space form $\Q_c^{n+1}$ with constant Ricci eigenvalues. If one of the following conditions is satisfied:
      \begin{itemize}
        \item[(1)] $M^n$ has at least three distinct constant Ricci eigenvalues;
        \item[(2)] $(n-1)c$ is not a Ricci eigenvalue of $M^n$;
        \item[(3)] the multiplicity of the Ricci eigenvalue $(n-1)c$ is not one of $\{n,n-2\}$,
      \end{itemize}
      then $M^n$ is isoparametric.
\end{corollary}

Combining Ryan's classification of complete Einstein hypersurfaces in spheres and Euclidean spaces \cite{Ryan1969} with Theorems \ref{thm:main}-\ref{thm:BFZhyper}, we have the complete version: 
\begin{corollary}\label{cor:S-isopara}
    Let $M^n$ be an $n(\geq 3)$-dimensional connected, complete hypersurface immersed in $\Sp^{n+1}$ with constant Ricci eigenvalues. Then $M^n$ is isoparametric.
\end{corollary}
\begin{corollary}\label{cor:R-isopara}
    Let $M^n$ be an $n(\geq 3)$-dimensional connected, complete nonflat hypersurface immersed in $\R^{n+1}$ with constant Ricci eigenvalues. Then $M^n$ is isoparametric.
\end{corollary}

For the hyperbolic case, we observe that non-isoparametric examples occur only in dimensions $n=3$ and $n=4$ in Theorems \ref{thm:BFZhyper} and \ref{thmH}. Thus we have the local version: 
\begin{corollary}\label{cor:H-isopara}
    Let $M^n$ be an $n(\geq 5)$-dimensional connected hypersurface immersed in $\Hi^{n+1}$ with constant Ricci eigenvalues. If $M^n$ is not of constant sectional curvature $-1$, then $M^n$ is isoparametric.
\end{corollary}

The rest of this paper is organized as follows.  In Section~\ref{sec2}, we present some preliminaries and reduce the problem to a study of the mean curvature function, showing in particular that its image is contained in the zero set of an explicitly constructed real-analytic function $F_\sigma$ (Lemma~\ref{lem:Image(H)}). By the identity theorem for real-analytic functions,  either $F_\sigma\equiv0$ or the image of the mean curvature function is discrete. In Section~\ref{sec3},  we first determine when $F_\sigma\equiv0$ (Proposition \ref{prop:F=0}); combining these preparations, we then apply the classification of curvature homogeneous hypersurfaces by Tsukada and Bryant-Florit-Ziller, and apply the classification of Einstein hypersurfaces by Ryan, to prove the main Theorem \ref{thm:main}, Theorem \ref{thm:cmc}, and Corollaries \ref{cor:Ric-isopara}--\ref{cor:H-isopara}.

\section{Preliminaries}\label{sec2}
Let $f \colon M^n \to \Q_c^{n+1}$ be an isometric immersion of a connected hypersurface $M^n$ into the simply connected real space form $\Q_c^{n+1}$ of constant sectional curvature $c$. Denote by $g$ the induced metric on $M^n$. For each point $p \in M$, the shape operator $A$ with respect to the unit normal vector field $\nu$ is a symmetric endomorphism of the tangent space $T_{p}M$.
Then the Gauss equation implies that
$$\oRic(X,Y)=(n-1)c\,g(X,Y)+ \otr(A)\,g(AX,Y)-g(A^2X,Y)$$
for all tangent vector fields $X,Y$ on $M$.  

For spectral considerations, it is often convenient to identify the $(0,2)$-tensor $\oRic$ with the corresponding $(1,1)$-tensor $\widehat\oRic$ by raising an index using the metric $g$, namely
$$g\bigl( \widehat\oRic(X), Y \bigr) = \oRic(X, Y), \quad \forall X, Y \in \mathfrak{X}(M).$$
Equivalently, the Ricci operator is given by
$$\widehat\oRic=(n-1)c\,Id+nHA-A^2,$$
where $H=\frac{1}{n}\otr(A)$ is the mean curvature of $M$. Since $\oRic$ is symmetric, the operator $\widehat\oRic$ is self-adjoint with respect to $g$. Hence, at each point $p\in M$, all eigenvalues of $\widehat\oRic_p \colon T_pM \to T_pM$ are real, and these eigenvalues are called the Ricci eigenvalues of $M$ at $p$.

Since $\widehat\oRic$ is a polynomial in $A$, the operators $\widehat\oRic$ and $A$ commute and can be simultaneously diagonalized at each point. Therefore, let $\{e_1,\cdots,e_n\}$ be an orthonormal basis of principal directions at an arbitrarily fixed point $p\in M$, so that
$$A(e_i)= \lambda_i e_i, \quad 1 \leq i \leq n,$$
then
$$\widehat\oRic(e_i)=\bigl((n-1)c+nH\lambda_i-\lambda_i^2\bigr)e_i,  \quad 1 \leq i \leq n.$$
Hence, the Ricci eigenvalues at $p$ are
$$\{(n-1)c+nH\lambda_i-\lambda_i^2,~ 1 \leq i \leq n\}.$$

Throughout this paper, we assume that $n \geq 3$ and $M^n$ has exactly $k~(1 \leq k \leq n)$ distinct constant Ricci eigenvalues, labeled by $\rho_1<\cdots<\rho_k$, with respective multiplicities $m_1,\cdots,m_k$, where $\sum_{j=1}^k m_j=n$. Hence, for every principal curvature $\lambda_i$, there exists
an index $j\in\{1,\cdots,k\}$ such that
\begin{equation}\label{eq:Gauss}
    \rho_j=(n-1)c+nH\lambda_i-\lambda_i^2.
\end{equation}
At each point of $M$, by equation \eqref{eq:Gauss}, every $\lambda_i$ corresponding to $\rho_j$ is a real root of the quadratic equation
\begin{equation}\label{eq:quadratic}
    x^2-nHx-c_j=0,
\end{equation}
where $c_j:=(n-1)c-\rho_j$ is constant. Therefore, for each fixed $j\in\{1,\cdots,k\}$, the possible principal curvatures corresponding to the Ricci eigenvalue $\rho_j$ are pointwise given by
\begin{equation}\label{root}
    \tilde\lambda_j^{\,\pm}(p)
    :=\frac12\Bigl(nH(p)\pm\sqrt{n^2H(p)^2+4c_j}\Bigr),
\end{equation}
whenever the discriminant $n^2H(p)^2+4c_j$ is nonnegative. In other words, if a principal direction
$e_i$ belongs to the eigenspace of the Ricci tensor associated with $\rho_j$, then its principal
curvature $\lambda_i(p)$ must be equal to one of the two values $\tilde\lambda_j^{\,+}(p)$ or
$\tilde\lambda_j^{\,-}(p)$.

\begin{remark}
    Notice that $\tilde\lambda_j^{\,\pm}$ are not the principal curvatures of $M$ themselves, but rather the two \emph{a priori} possible root values determined by the Ricci eigenvalue $\rho_j$. The actual principal curvatures of $M$ are the $\lambda_i$'s, counted with multiplicities, obtained by selecting pointwise, for each $j$, one or both of the values $\tilde\lambda_j^{\,+}$ and $\tilde\lambda_j^{\,-}$ according to the corresponding multiplicity distribution in the Ricci eigenspace of $\rho_j$.
\end{remark}

\begin{lemma}\label{lem:degeneracy}
    For every $j\in\{1,\cdots,k\}$ one has
    $$n^2H^2+4c_j\ge 0 \quad \text{on } M.$$
    Moreover, if $n^2H^2+4c_j=0$ at some point of $M$, then necessarily $j=k$.
\end{lemma}
\begin{proof}
    For each $j$, the quantity $n^2H^2+4c_j$ is the discriminant of the quadratic equation \eqref{eq:quadratic} satisfied by the corresponding principal curvatures. Since principal curvatures are real on $M$, the discriminant must be nonnegative at every point, i.e., $n^2H^2+4c_j\ge 0$ on $M$.

    The ordering $c_1>\cdots>c_k$ follows immediately from the definition $c_j=(n-1)c-\rho_j$ and the ordering of the Ricci eigenvalues. If $n^2H^2+4c_j=0$ for some $j<k$, then
    $$0=n^2H^2+4c_j>n^2H^2+4c_\ell \quad \text{for all } \ell>j,$$ which is impossible. Therefore, degeneracy can occur only for the smallest $c_k$.
\end{proof}

For each fixed $j\in\{1,\cdots,k\}$, let $\sigma_j:M \to \R$ denote the pointwise signature, defined as the number of positive branches $\tilde\lambda_j^{\,+}$ minus the number of negative branches $\tilde\lambda_j^{\,-}$ among the $m_j$ principal curvatures \eqref{root}. For points $q \in M$ where the two roots coincide, by convention, we define $\sigma_j(q)=0$. Let $\mathcal A_j:=\{-m_j,-m_j+2,\cdots,m_j-2,m_j\}$. Then the following lemma holds.

\begin{lemma}\label{lem:delta}
    For every $j\in\{1,\cdots,k-1\}$, $\sigma_j$ is constant on $M$ and takes a value in $\mathcal A_j$. $\sigma_k$ is constant on any connected component of $M$ where $n^2H^2+4c_k>0$ and $\sigma_k \in \mathcal A_k \cup \{0\}$ on $M$.
\end{lemma}
\begin{proof}
    For $1 \leq j \leq k-1$,  the two possible principal curvatures $\tilde\lambda_j^{\,\pm}$ given in \eqref{root}, are distinct on $M$ by Lemma~\ref{lem:degeneracy}. For each principal direction, switching between the two algebraic roots $\tilde\lambda_j^{\,\pm}$ from point to point is impossible, since the principal curvature functions are continuous on the connected hypersurface $M$. Thus, one must continuously choose the same branch of the quadratic solution for each principal direction, which determines a constant sign in the superscript of $\tilde\lambda_j^{\,\pm}$, and thus $\sigma_j \in \mathcal A_j$ is constant on $M$. Moreover, if there exist points $q \in M$ at which $n^2H^2(q)+4c_k=0$, then the corresponding two roots $\tilde\lambda_k^{\,\pm}(q)$ coincide, $\sigma_k(q)$ vanishes, and the local sign in the superscript of $\tilde\lambda_k^{\,\pm}$ may change when passing through such degeneracy points $q$, while $\sigma_k \in \mathcal A_k$ remains constant on each connected component where $n^2H^2+4c_k>0$ by continuity and Lemma~\ref{lem:degeneracy}. Hence, $\sigma_k \in \mathcal A_k \cup \{0\}$ on $M$.
\end{proof}

Denote by $\mathcal A:=(\prod_{j=1}^{k-1} \mathcal A_j)\times(\mathcal A_k \cup \{0\})$ the finite set of admissible vectors. Then by Lemma \ref{lem:delta}, the image of the piecewise constant vector-valued function $\sigma:=(\sigma_1,\cdots,\sigma_k)$ on $M$ lies in $\mathcal{A}$.
Let $D:=\{h\in\R \colon n^2h^2+4c_k>0\}$. It is easy to see that $D=\mathbb{R}$ if $c_k>0$, $D=\mathbb{R}\setminus \{0\}$ if $c_k=0$, and $D=\mathbb{R}\setminus [-\frac{2}{n}\sqrt{|c_k|}, \frac{2}{n}\sqrt{|c_k|}]$ if $c_k<0$. By Lemma~\ref{lem:degeneracy}, the mean curvature $H$ is a continuous function from $M$ to $\overline{D}$.
For every $\delta:=(\delta_1,\cdots,\delta_k)\in\mathcal A$, we define $F_\delta \colon \overline{D} \to \R$ as
\begin{equation}\label{eq:F}
    F_\delta(h):=n(n-2)h+\sum_{j=1}^k \delta_j\sqrt{n^2h^2+4c_j}.
\end{equation}

\begin{lemma}\label{lem:Image(H)}
    Let $M^n$ be an $n(\geq 3)$-dimensional connected hypersurface immersed in a real space form $\Q_c^{n+1}$ with constant Ricci eigenvalues. Then  $F_\sigma \circ H=0$ on $M$, i.e., $\oImag(H) \subset \bigcup_{\delta\in \oImag(\sigma)}F^{-1}_{\delta}(0)$.
\end{lemma}
\begin{proof}
   Summing all admissible principal curvatures in representation \eqref{root} and grouping together those associated with the same Ricci eigenvalue, one obtains
    $$nH=\frac{n^2}{2}H+\frac{1}{2}\sum_{j=1}^{k}\sigma_j \sqrt{n^2H^2+4c_j},$$
    which is equivalent to
    \begin{equation}\label{eq:H}
        n(n-2)H+\sum_{j=1}^{k}\sigma_j \sqrt{n^2H^2+4c_j}=0.
    \end{equation}
     The proof is completed by combining the expressions~\eqref{eq:F} and \eqref{eq:H}.
\end{proof}

Hence, we reduce the classification problem to the investigation on the mean curvature function, whose image is contained in the zero set of an explicitly constructed function $F_\delta$. If $I$ is a connected component of $D$, then $F_\delta$ is a real-analytic function on $I$. It follows from the identity theorem for real-analytic functions that its zero set is either discrete or the whole component. The structure of the zero set is analyzed in more detail in the next section (Propositions~\ref{prop:F=0} and \ref{prop:F!0}).


\section{Proof of the Main Results}\label{sec3}
In this section, we aim to prove the classification of hypersurfaces with constant Ricci eigenvalues in real space forms (Theorem~\ref{thm:main}) and characterize when such hypersurfaces are isoparametric (Theorem~\ref{thm:cmc} and Corollaries \ref{cor:Ric-isopara}-\ref{cor:H-isopara}). Motivated by Lemma~\ref{lem:Image(H)}, we first investigate the zero set of $F_\delta$ and establish the following two propositions.

\begin{proposition}\label{prop:F=0}
     Let $I_+\subset (0, +\infty)~(\text{resp. } I_-\subset (-\infty,0))$ be an unbounded connected component of $D$ and $n\ge 3$. For  $\delta\in\mathcal A$, 
     $F_\delta\equiv 0$ on $I_+~(\text{resp. } I_-)$ if and only if there exists a unique index $j_0\in\{1,\cdots,k\}$ such that
     $$c_{j_0}=0,\quad \delta_{j_0}=2-n~(\text{resp. }\delta_{j_0}=n-2),\quad \delta_j=0 \text{ for } j\neq j_0.$$
\end{proposition}
\begin{proof}
We prove the positive case first; the proof of the negative case is analogous.

\medskip
\emph{Necessity.}
Assume that $F_\delta\equiv 0$ on a positive unbounded connected component $I_+$. As $h\to+\infty$ within $I_+$, for each $j$, one has 
$$\sqrt{n^2h^2+4c_j}=nh\sqrt{1+\frac{4c_j}{n^2h^2}}=
\sum_{m=0}^\infty \binom{1/2}{m}(4c_j)^m (nh)^{-2m+1},$$
where $$\binom{1/2}{m}=\frac{\frac{1}{2}(\frac{1}{2}-1) \cdots (\frac{1}{2}-m+1)}{m!}=\frac{(-1)^{m-1}(2m-3)!!}{2^m m!}.$$
Substituting this into $F_\delta(h)\equiv 0$ yields
$$F_\delta(h)=
\Bigl((n-2)+\sum_{j=1}^k\delta_j\Bigr)nh+\sum_{m=1}^\infty  (-1)^{m-1}2^m\frac{(2m-3)!!}{m!}\sum_{j=1}^k \delta_j c_j^m (nh)^{-2m+1} \equiv 0,
$$
and comparing coefficients of the powers of $nh$ gives
\begin{equation}\label{coefficients}
     \sum_{j=1}^k \delta_j=-(n-2),
     \quad
     \sum_{j=1}^k \delta_j c_j^m=0 \text{ for all } m\ge 1.
\end{equation}

Since the numbers $c_1,\cdots,c_k$ are pairwise distinct, each value among $\{c_j\}_{j=1}^{k}$ occurs exactly once. Hence, restricting the second family of relations in \eqref{coefficients} to the nonzero $c_j$, we obtain a homogeneous linear system with a Vandermonde matrix, which is invertible. Therefore,
$\delta_j=0 \text{ whenever } c_j\neq 0.$
It remains to analyze the contribution of the zero value among $\{c_j\}_{j=1}^{k}$. The first relation in \eqref{coefficients} forces the existence of a unique index $j_0$ with $c_{j_0}=0$, and then $\delta_{j_0}=2-n.$




\medskip
\emph{Sufficiency.}
Conversely, assume that there exists $j_0 \in\{1,\cdots,k\}$ such that $c_{j_0}=0$, $\delta_{j_0}=2-n$ and $\delta_j=0$ for $j\neq j_0$.
For arbitrary $h\in I_+$, we have $\sqrt{n^2h^2}=nh$. It follows that
$$F_\delta(h)=n(n-2)h+\delta_{j_0}\sqrt{n^2h^2}
=n(n-2)h+(2-n)nh=0$$
for all $h\in I_+$, so that $F_\delta\equiv 0$ on $I_+$.

\medskip
Similarly, we prove the negative case. For the \emph{necessity}, using instead the expansion
$$\sqrt{n^2h^2+4c_j}
=-\sum_{m=0}^\infty \binom{1/2}{m}(4c_j)^m (nh)^{-2m+1},
\quad \text{as } h\to-\infty,$$
which leads to
$$F_\delta(h)=
\Bigl((n-2)-\sum_{j=1}^k\delta_j\Bigr)nh-\sum_{m=1}^\infty  (-1)^{m-1}2^m\frac{(2m-3)!!}{m!}\sum_{j=1}^k \delta_j c_j^m (nh)^{-2m+1} \equiv 0,
$$
and thus
\begin{equation*}\label{coefficients-}
     \sum_{j=1}^k \delta_j=n-2,
     \quad
     \sum_{j=1}^k \delta_j c_j^m=0 \text{ for all } m\ge 1.
\end{equation*}
The same Vandermonde matrix as above shows that $\delta_j=0$ whenever $c_j\neq 0$, and hence there exists a unique index $j_0$ such that $c_{j_0}=0,~\delta_{j_0}=n-2$. For the \emph{sufficiency}, we have $\sqrt{n^2h^2}=-nh$ for $h\in I_-$, and hence 
$$F_\delta(h)=n(n-2)h+\delta_{j_0}\sqrt{n^2h^2}
=n(n-2)h-(n-2)nh=0$$
for all $h\in I_-$, so that $F_\delta\equiv 0$ on $I_-$. The proof is now complete.
\end{proof}

\begin{remark}
    If $D$ consists of exactly one connected component $I$, i.e., $c_k>0$ and $I=D=\R$, then $F_\delta \equiv 0$ on $I$ is impossible by Proposition~\ref{prop:F=0}. 
\end{remark}


\begin{proposition}\label{prop:F!0}
     Let $I$ be a connected component of $D$. If $F_\delta\not\equiv0$ on $I$ for some $\delta \in \mathcal{A}$, then $F^{-1}_{\delta}(0)\cap I$ is discrete. Consequently, for a hypersurface $M^n$ with constant Ricci eigenvalues in a real space form $\Q_c^{n+1}$, on every connected component $U$ of $M$ where $n^2H^2+4c_k>0$ (i.e., $\oImag(H|_U)\subset D$), either $F_{\sigma}(h)\equiv0$ on the connected component in $D$ of $h\in \oImag(H|_U)\subset D$ or the mean curvature $H$ is constant.
\end{proposition}
\begin{proof}
     By the expression \eqref{eq:F}, $F_\delta$ is an analytic function of $h$ on any connected component $I$ where the square roots are fixed as real-analytic branches. Hence, its zero set $F^{-1}_{\delta}(0)\cap I$ is discrete by the identity theorem for real-analytic functions under the assumption that $F_\delta\not\equiv 0$ on $I$.
     
      On every connected component $U$ of $M$ where $n^2H^2+4c_k>0$, the admissible vector function $\sigma\in\mathcal{A}$ is constant, and $\oImag(H|_U)\subset D$ is located in exactly one connected component of $D$ since the mean curvature $H$ is a continuous function on $M$. Combined with Lemma~\ref{lem:Image(H)}, we conclude that $\oImag(H|_U) \subset F^{-1}_{\sigma}(0)$ is either a single point (thus $H$ is constant), or an interval and $F_{\sigma}(h)\equiv0$ on its connected component in $D$. 
\end{proof}

Based on the \emph{a priori} possible root values \eqref{root} for principal curvatures, we derive a criterion (Theorem~\ref{thm:cmc}) that determines which hypersurfaces with constant Ricci eigenvalues are isoparametric. This criterion will be used later in the proof of Theorem~\ref{thm:main}.

\begin{proof}[Proof of Theorem~\ref{thm:cmc}]
    The \emph{necessity} is obvious from the definition. We now prove the \emph{sufficiency}.
    Suppose that the mean curvature $H$ of $M^n$ is constant. At every point of $M$, by \eqref{root}, the principal curvatures must belong to the finite set
    $$\mathcal{S}:=\left\{\frac{nH\pm\sqrt{n^2H^2+4c_j}}{2}:\ j=1,\cdots,k\right\}.$$
    Since every principal curvature function $\lambda_i \colon M \to \R$ is continuous on connected $M$, $\oImag(\lambda_i)$ is also connected and therefore degenerates to a single point in $\mathcal{S}$. Thus, all principal curvatures are constant on $M$, which means that $M$ is isoparametric.
\end{proof}

Applying Lemma~\ref{lem:Image(H)}, Propositions \ref{prop:F=0}, \ref{prop:F!0}, together with Theorem~\ref{thm:cmc}, we now characterize hypersurfaces in real space forms with constant Ricci eigenvalues.

\begin{proof}[Proof of Theorem~\ref{thm:main}]
    \emph{Necessity.} 
    As curvature homogeneity is a local property, we need only prove it locally.
    If $H$ is constant, then $M$ is already isoparametric by Theorem~\ref{thm:cmc}, and hence curvature homogeneous by Tsukada's characterization. From now on, assume that $H$ is not constant, thus $n^2H^2+4c_k\not\equiv0$ and $\oImag(H)$ is an interval. By Lemma \ref{lem:degeneracy}, we have two cases: either $n^2H^2+4c_k>0$ everywhere, or the subset of $M$ where $n^2H^2+4c_k>0$ is a nonempty open submanifold of $M$ with countably many connected components $U_\alpha$ with $n^2H^2+4c_k=0$ on $M\setminus \bigcup_\alpha U_\alpha$. 
    
    For the first case, by Lemma \ref{lem:delta} $\sigma\in\mathcal{A}$ is a constant vector and by Proposition \ref{prop:F!0}, $F_{\sigma}(h)\equiv0$ on the connected component in $D$ of $h\in \oImag(H)\subset D$. 
        By Proposition~\ref{prop:F=0}, there exists a unique index $j_0$ such that $c_{j_0}=0, ~|\sigma_{j_0}|=n-2, \text{ and }  \sigma_j=0 \text{ for } j\neq j_0.$
    Thus $\rho_{j_0}=(n-1)c$. Since $\sigma_{j_0}\in \mathcal{A}_{j_0} \cup \{0\}$, the multiplicity of $\rho_{j_0}$ must satisfy $n-2 \leq m_{j_0} \leq n$ and $m_{j_0} \equiv n-2 \pmod 2$. It follows that $m_{j_0} \in \{n-2,n\}$.

    If $m_{j_0}=n$, then every principal curvature is a root of $x^2-nHx=0$, hence is either $0$ or $nH$. Since the sum of all principal curvatures equals $nH$, exactly one of them is $nH$ and the remaining $(n-1)$ are $0$. Therefore, the spectrum of the shape operator of $M$ is
    $$\oSpec(A)=\{nH,0,\cdots,0\}.$$ 
    By the Gauss equation, $M$ has constant sectional curvature $c$ and hence $M$ is curvature homogeneous by Tsukada's characterization. 

    If $m_{j_0}=n-2$, then the sum of the remaining multiplicities is $2$. Since $\sigma_j=0$ for all $j\neq j_0$, each $m_j$ with $j\neq j_0$ must be even; hence there is exactly one index $j_1\neq j_0$ with $m_{j_1}=2$. Thus $k=2$ and exactly one Ricci eigenvalue $\rho_{j_0}=(n-1)c$ with multiplicity $n-2$. Let $\mu,\nu$ be the two principal curvatures corresponding to $\rho_{j_1}$. Then, by \eqref{eq:quadratic}, we have
    $$\mu+\nu=nH, \quad \mu\nu=-c_{j_1}=\rho_{j_1}-(n-1)c\neq0.$$
    The remaining $(n-2)$ principal curvatures solve $x^2-nHx=0$, so each is either $0$ or $nH$. Since $\mu+\nu=nH$, the sum of those $(n-2)$ principal curvatures is $0$, and therefore all of them vanish. Hence, 
    $$\oSpec(A)=\{\mu,\nu,0,\cdots,0\}.$$
    Since $\mu\nu \neq 0$, $M$ is a rank-two hypersurface with constant scalar curvature, and hence $M$ is curvature homogeneous by Tsukada's characterization. 

    For the second case, on each connected component $U_\alpha$ where $n^2H^2+4c_k>0$, by Proposition \ref{prop:F!0}, we have that either $H$ is constant or $F_{\sigma}(h)\equiv0$ on the connected component in $D$ of $h\in \oImag(H|_{U_\alpha})\subset D$, which implies that $U_\alpha$ is one of the same three kinds of curvature homogeneous hypersurfaces as above: isoparametric; constant sectional curvature $c$; rank-two with constant scalar curvature. On the interior $\Omega$ (if it exists) of the subset $M\setminus \bigcup_\alpha U_\alpha$ where $n^2H^2+4c_k=0$, $H$ is constant and thus it is isoparametric by Theorem~\ref{thm:cmc}. For points on the boundary of $U_\alpha$, the required isometry $J_{pq}$ can be obtained from that of $U_\alpha$ or $\Omega$ since the curvature homogeneity (i.e., the three kinds of properties) is continuous and closed under taking limits. In conclusion, $M$ is again curvature homogeneous.

    \medskip
    \emph{Sufficiency.} This is obvious by definition, so we give a brief proof here.
     Let $M$ be a curvature homogeneous manifold. For any two points $p,q\in M$, there exists a linear isometry $J_{pq} \colon T_pM\to T_qM$ such that $$R_q\bigl(J_{pq}X,J_{pq}Y,J_{pq}Z,J_{pq}W\bigr)=R_p(X,Y,Z,W)$$ for all $X,Y,Z,W\in T_pM$. Since the Ricci tensor is obtained from the curvature tensor by metric contraction, one has $\oRic(X,Y)=\sum_{i=1}^n R(e_i,X,Y,e_i)$ for any orthonormal basis $\{e_i\}_{i=1}^n$. It follows that $$\oRic_q\bigl(J_{pq}X,J_{pq}Y\bigr)=\oRic_p(X,Y)$$ for all $X,Y\in T_pM$.
    Equivalently, if $\widehat{\oRic}_p$ and $\widehat{\oRic}_q$ denote the Ricci endomorphisms at $p$ and $q$, respectively, then $$\widehat{\oRic}_q\circ J_{pq}=J_{pq}\circ \widehat{\oRic}_p.$$
    Hence $\widehat{\oRic}_p$ and $\widehat{\oRic}_q$ are orthogonally conjugate, and therefore have the same characteristic polynomial. In particular, their eigenvalues, counted with multiplicities, coincide. Since $p$ and $q$ are arbitrary, the spectrum of the Ricci endomorphism is independent of the point. Therefore, the Ricci eigenvalues are constant on $M$.
\end{proof}

At last, we prove the sufficient conditions (Corollaries \ref{cor:Ric-isopara}--\ref{cor:H-isopara}) for a hypersurface with constant Ricci eigenvalues to be isoparametric.
It suffices to exclude the two kinds of curvature homogeneity (constant sectional curvature $c$; rank-two with constant scalar curvature) from the isoparametric kind.

\begin{proof}[Proof of Corollary~\ref{cor:Ric-isopara}]
      By Theorems \ref{thm:main} and \ref{thm:cmc}, the characterization of isoparametric hypersurfaces is reduced to proving the constancy of the mean curvature $H$, and the main issue is to exclude the possibility that $F_\sigma$ vanishes identically on a connected component of $D$ (i.e., the rank of the shape operator of $M$ is at most two). For the excluded cases, by the proof of Theorem~\ref{thm:main}, $M$ must have a Ricci eigenvalue $(n-1)c$ with multiplicity $n$ or $(n-2)$. The corresponding number of distinct Ricci eigenvalues is $1$ or $2$, respectively. Hence, each one of the three sufficient conditions stated in the corollary excludes these cases and thus implies the constancy of the mean curvature $H$. This completes the proof.
\end{proof}

\begin{proof}[Proof of Corollary~\ref{cor:S-isopara}]
       As Ryan \cite{Ryan1969} showed, the complete Einstein hypersurfaces in spheres are the small spheres, the great spheres and certain products of spheres, which are isoparametric hypersurfaces with $1$ or $2$ distinct principal curvatures. Hence, the kind of constant sectional curvature $1$ is automatically isoparametric as it is Einstein.
       For the kind of rank-two with constant scalar curvature, the conclusion follows from Theorem \ref{thmS} for $n\geq4$, and from Theorem \ref{thm:BFZhyper} for $n=3$ where $f_1$ is the only complete example and it is exactly the Cartan isoparametric hypersurface with $3$ distinct principal curvatures. 
\end{proof}

\begin{proof}[Proof of Corollary~\ref{cor:R-isopara}]
      As Ryan \cite{Ryan1969} showed, the complete Einstein hypersurfaces in Euclidean spaces are spheres (thus isoparametric), hyperplanes and cylinders over complete plane curves which are flat (thus excluded by the nonflat assumption). For the kind of rank-two with constant scalar curvature, the conclusion follows from the third case in Theorem \ref{thmR}. Indeed, if $M^n$ is complete, the local product structure lifts to the universal cover and becomes a global product. Hence the two-dimensional factor $M^2_{\kappa}$ is complete and, again by Ryan's classification of complete Einstein hypersurfaces in Euclidean spaces, it must be a sphere. Therefore, the third case reduces to the spherical cylinders $\Sp^2(\frac{1}{\sqrt{\kappa}})\times \R^{n-2}$, which are isoparametric hypersurfaces with $2$ distinct principal curvatures.
      
\end{proof}

\begin{proof}[Proof of Corollary~\ref{cor:H-isopara}]
      For $n\geq5$, by Theorem \ref{thmH}, the curvature homogeneous hypersurfaces in hyperbolic spaces are either of constant sectional curvature $-1$, or totally umbilical (which is isoparametric with $1$ principal curvature), or locally congruent to $\Sp^r_{c_1}\times \Hi^{n-r}_{c_2}$ (which are isoparametric hypersurfaces with $2$ distinct principal curvatures).
\end{proof}

	\bibliographystyle{plain}

\end{document}